\documentclass[12pt, reqno, oneside]{amsart}
\usepackage{amsmath}
\usepackage{amsthm}
\usepackage{amssymb}
\usepackage[style=alphabetic, maxnames=50]{biblatex}
\usepackage{mathdots}
\usepackage{mathtools}
\usepackage{graphicx}
\usepackage{hyperref}
\usepackage{xcolor}
\hypersetup{
    colorlinks=true,
    linkcolor=cyan,      
    pdfpagemode=FullScreen,
    citecolor=orange,
}
\newtheorem{corollary}{Corollary}[section]
\newtheorem{definition}[corollary]{Definition}
\newtheorem{example}[corollary]{Example}
\newtheorem{lemma}[corollary]{Lemma}
\newtheorem{proposition}[corollary]{Proposition}

\newtheorem{theorem}[corollary]{Theorem}
\numberwithin{equation}{section}

\allowdisplaybreaks[3]

\title[Explicit formulas for Casimir eigenvalues of Maass forms]{Explicit formulas for the Casimir eigenvalues of $SL(n,\mathbb{Z})$-Maass forms}
\author{Vishal Muthuvel}
\address{Department of Mathematics, Columbia University, New York, NY 10027}
\email{\href{mailto:vm2696@columbia.edu}{vm2696@columbia.edu}}
\date{\today}

\begin{document}
\begingroup

%
%

\begin{abstract}

Maass forms for $SL(n,\mathbb{Z})$ are defined to be eigenfunctions of the Casimir operators $\mathcal{D}_{m,n}$ of orders $1 \leq m \leq n$ for $GL(n,\mathbb{R})$. For any $1 \leq m \leq n$ and Maass form $\phi$ for $SL(n,\mathbb{Z})$, we provide a formula for the eigenvalue of $\mathcal{D}_{m,n}$ associated with $\phi$ in terms of the Langlands parameters of $\phi$. In the case $m=2$, we recover the formula for the Laplace eigenvalue of a Maass form due to Terras, the Casimir differential operator of order $2$ being the Laplacian. Our proof takes a graph-theoretic approach, relating the action of every elementary differential operator of order $m$ for $GL(n,\mathbb{R})$ to the partitions of a directed, edge-ordered, weighted graph with $m$ edges and at most $m$ vertices. 

\end{abstract}

\maketitle

\vspace*{-0.75cm}

\tableofcontents

\vspace*{-1.25cm}

%
%

\section{Introduction}

Let $n \geq 2$. Consider the Lie group $G = GL(n,\mathbb{R})$ with maximal compact subgroup $K = O(n,\mathbb{R})$ and center $Z(G) \cong \mathbb{R}^\times$. We study automorphic forms for $\Gamma_n = SL(n,\mathbb{Z})$. These are elements of 
\begin{equation*}
\mathcal{L}^2 \left( \Gamma_n \backslash G / (K \cdot Z(G)) \right),
\end{equation*}
i.e., functions on the generalized upper half-plane
\begin{equation*}
\mathfrak{h}^n \ := \ G / (K \cdot Z(G))
\end{equation*}
that are invariant by $\Gamma_n$ on the left and square-integrable under the $G$-invariant measure on $\mathfrak{h}^n$. In particular, these are joint eigenfunctions of the $G$-invariant differential operators acting on $\mathcal{L}^2 \left( \Gamma_n \backslash \mathfrak{h}^n \right)$. 

\sloppy Let $\mathfrak{g} = \mathfrak{gl}(n,\mathbb{R})$. This is the Lie algebra associated to $G$, consisting of all $n \times n$ matrices with coefficients in $\mathbb{R}$. Casimir operators for $G$ are generators for the center of the universal enveloping algebra of $\mathfrak{g}$
\begin{equation*}
Z(U(\mathfrak{g})),
\end{equation*}
which contains precisely the $G$-invariant differential operators acting on $\mathcal{L}^2 \left( \Gamma_n \backslash \mathfrak{h}^n \right)$ (see Section \ref{sec:lie-theory}). For each $1 \leq m \leq n$, there is a Casimir operator of order $m$ for $G$; the Casimir operator of order $2$ is notably the Laplacian. 

The actions of the Casimir and Hecke operators on $\mathcal{L}^2 \left( \Gamma_n \backslash \mathfrak{h}^n \right)$ commute and simultaneously diagonalize the space of automorphic forms for $\Gamma_n$. Selberg \cite{Sel56} obtained the decomposition of the joint spectrum of these operators when $n = 2$ using his famous trace formula. Arthur and Langlands \cite{Art79, Lan76} generalized Selberg's spectral decomposition to all $n$, showing that the joint spectrum decomposes into a continuous, a cuspidal, and a residual part. The continuous spectrum is uncountable and comes from Eisenstein series associated to the parabolic subgroups of $G$, the residual spectrum comes from the residues of these Eisenstein series, and the cuspidal spectrum is countable and comes from Maass forms for $\Gamma_n$. In this paper, we consider the actions of the Casimir operators on Maass forms for $\Gamma_n$. 

By the standard adelic lifting, Maass forms for $\Gamma_n$ correspond to cuspidal automorphic representations
\begin{equation*}
\pi \ = \ \bigotimes_v \pi_v
\end{equation*}
of $GL(n,\mathbb{Z}_\mathbb{Q})$, where the product is over places $v$ that are either finite primes or $\infty$. Each of these places is described by certain parameters. The infinite place, in particular, is described by a list of $n$ complex numbers 
\begin{equation}
\alpha = (\alpha_1, \alpha_2, \dots, \alpha_n) \in \mathbb{C}^n,
\end{equation}
called the Langlands parameters, which satisfy $\sum_{\ell=1}^n \alpha_\ell = 0$. The $\alpha_\ell$ familiarly appear in the Gamma factors of the completed $L$-function associated to $\pi$. 

We provide explicit formulas for the eigenvalues of Casimir operators associated with Maass forms for $\Gamma_n$, in terms of the Langlands parameters of the Maass forms. Our formulas can be readily generalized to minimal parabolic (Borel) Eisenstein series as well.

\begin{theorem}\label{thm:opening}
Let $1 \leq m \leq n$. Let $\mathcal{D}_{m,n}$ be the Casimir operator of order $m$ for $GL(n,\mathbb{R})$, and let $\phi$ be a Maass form for $SL(n,\mathbb{Z})$ with Langlands parameters $(\alpha_1, \dots, \alpha_n) \in \mathbb{C}^n$. The eigenvalue of the Casimir operator of order $m$ for $GL(n,\mathbb{R})$ associated with $\phi$ is given by
\begin{align}\label{eq:casimir-eigenvalues}
\mathcal{D}_{m,n} \phi \ = \ 
\begin{cases} 
0 & \text{if } m=1, \\
\left(\sum_{\ell=1}^n \alpha_\ell^2 - \frac{n^3-n}{12}\right)\phi & \text{if } m=2, \\
\left(\sum_{\ell=1}^n \alpha_\ell^3 - \frac{n}{2}\sum_{\ell=1}^n \alpha_\ell^2 + \frac{n^4-n^2}{24}\right)\phi & \text{if } m=3, \\
\left(\sum_{\ell=1}^n \alpha_\ell^4 - n \sum_{\ell=1}^n \alpha_\ell^3 + \frac{1}{2} \sum_{\ell=1}^n \alpha_\ell^2 - \frac{n^5-n}{80}\right)\phi & \text{if } m=4.
\end{cases}
\end{align}
\end{theorem}

The Casimir operator of order $2$ for $G$ is the Laplacian $\Delta_n$ for $\mathfrak{h}^n$, in particular,
\begin{equation}
\mathcal{D}_{2,n} \ = \ \frac{1}{2} \Delta_n,
\end{equation}
and the formula for the Laplace eigenvalue of a Maass form was first given by Terras \cite{Ter88}. For Casimir operators of order $m \geq 3$, no such formulas have appeared in the literature before. In addition to providing these for $m \leq 4$ in \eqref{eq:casimir-eigenvalues}, we outline an algorithm that can be used to find such formulas for any $m$ in Theorem \ref{thm:elementary-differential-operator}. 

The Casimir operator of order $m$ for $G$ is defined as the sum of all elementary differential operators of order $m$, each corresponding to an $m$-tuple $(i_1,i_2,\dots,i_m) \in \{1, \dots, n\}^m$:
\begin{equation}\label{eq:casimir-operator}
\mathcal{D}_{m,n} \ := \ \sum_{(i_1,\dots,i_m) \in \{1, \dots, n\}^m} D_{i_1,i_2} \circ \cdots \circ D_{i_{m-1},i_m} \circ D_{i_m,i_1}.    
\end{equation}
(See Definition \ref{def:casimir-operator} for a complete description.) For arbitrary $(i_1,\dots,i_m) \in \{1, \dots, n\}^m$, we provide explicit formulas for the eigenvalues of $D_{i_1,i_2} \circ \cdots \circ D_{i_{m-1},i_m} \circ D_{i_m,i_1}$ associated with the power function (see Definition \ref{def:power-function}). To present them clearly, we require the following definitions.

Consider an ordered list of $k \geq 1$ integers: 
\begin{equation*}
T = (t_j)_{1 \leq j \leq k} \in \mathbb{Z}^k.    
\end{equation*}

\begin{definition}
We call 
\begin{align}
v_1(T) \ &:= \ \min\left(\{t_j\}_{1 \leq j \leq k}\right), \\
v_2(T) \ &:= \ 
\begin{cases}
\min\left(\{t_j\}_{1 \leq j \leq k} \backslash \{v_1(T)\}\right) & \text{if } \ \{t_j\}_{1 \leq j \leq k} \backslash \{v_1(T)\} \neq \emptyset \\
\infty & \text{if } \ \{t_j\}_{1 \leq j \leq k} \backslash \{v_1(T)\} = \emptyset 
\end{cases}
\end{align}
the first and second minimum elements of $T$, respectively. 
\end{definition}

\begin{definition}[sub-list]
We call $S$ a sub-list of $T$ and write $S \subset T$ if $S = (t_j)_{\ell \leq j \leq \ell'}$ for some $1 \leq \ell \leq \ell' \leq k$.
\end{definition}

\begin{definition}[cycle]\label{def:cycle}
We call $T$ a cycle on $c \in \mathbb{Z}$ if $t_1 = t_k =c$. Additionally, if $t_1 < t_\ell$ for all $2 \leq \ell \leq k-1$, then we call $T$ a proper cycle. 
\end{definition}

\begin{theorem}\label{thm:elementary-differential-operator}
Consider any $1 \leq m \leq n$, and any $(i_1, i_2, \dots, i_m) \in \{1, \dots, n\}^m$. Let $I = (i_1, \dots, i_m, i_1)$, and let $|\cdot|^\alpha$ be the power function for $GL(n,\mathbb{R})$ with Langlands parameters $\alpha = (\alpha_1,\dots,\alpha_{n}) \in \mathbb{C}^n$. Adopting the convention $\alpha_\infty = 0$, the eigenvalue of $D_{i_1,i_2} \circ \cdots \circ D_{i_{m-1},i_m} \circ D_{i_m,i_1}$ associated with $|\cdot|^\alpha$ is given by
\begin{align}\label{eq:elementary-eigenvalue}
\nonumber
& D_{i_1,i_2} \circ \cdots \circ D_{i_{m-1},i_m} \circ D_{i_m,i_1} |\cdot|^\alpha \\
& = \ |\cdot|^\alpha
\begin{cases}
0 \\
\text{if } \ i_1 > i_\ell \ \text{ for some } \ 1 \leq \ell \leq m, \\
\prod\limits_{\substack{S \subset I \text{ proper} \\ \text{cycle on } i_1}}\left(-\alpha_{v_1(S)}+\alpha_{v_2(S)}\right)\prod\limits_{\substack{S \subset I \text{ proper} \\ \text{cycle on } c > i_1}}\left(-\alpha_{v_1(S)}+\alpha_{v_2(S)} + 1\right) \\
\text{if } \ i_1 \leq i_\ell \ \text{ for all } \ 1 \leq \ell \leq m.
\end{cases} 
\end{align}
\end{theorem}

To demonstrate this theorem in practice, we present an example. 

\begin{example}
Let $m=n=10$, and 
\begin{equation*}
(i_1,i_2,\dots,i_{10}) \ = \ (1,9,2,5,5,9,6,8,4,5).    
\end{equation*}
To compute the eigenvalue of $D_{i_1,i_2} \circ \cdots \circ D_{i_{10},i_1}$, we identify all sub-lists of $I = (i_1,\dots,i_{10},i_1)$ that are proper cycles. To this end, we note that every sub-list of $I$ that is a cycle arises from a repeated element of $I$:
\begin{equation*}
i_1 = i_{11} = 1, \ i_2 = i_6 = 9, \ i_4 = i_5 = 5, \ i_5 = i_{10} = 5.   
\end{equation*}
We record the cycles corresponding to these repeated elements in the following table. 
\begin{table}[hbt!]
    \centering
    \begin{tabular}{|c|c|c|c|}
    \hline
        Cycle $(S)$ & Proper? & $v_1(S)$ & $v_2(S)$ \\
        \hline
        $(i_1,i_2\dots,i_{10},i_1)=(1,9,2,5,5,9,6,8,4,5,1)$ & Yes & $i_1 = 1$ & $i_3 = 2$ \\
        $(i_2,i_3\dots,i_6) = (9,2,5,5,9)$ & No & $i_3 = 2$ & $i_4 = 5$ \\
        $(i_4,i_5)=(5,5)$ & Yes & $i_4 = 5$ & $\infty$ \\
        $(i_5,i_6,\dots,i_{10}) =  (5,9,6,8,4,5)$ & No & $i_9 = 4$ & $i_5 = 5$ \\
        \hline
    \end{tabular}
    \vspace{2mm}
    \caption{Cycles $S \subset (i_1,\dots,i_{10},i_1) = (1,9,2,5,5,9,6,8,4,5,1)$}
    \label{tab:example-cycles}
\end{table}

Therefore, the eigenvalue of $D_{i_1,i_2} \circ \cdots \circ D_{i_9,i_{10}} \circ D_{i_{10},i_1}$ associated with $|\cdot|^\alpha$ is
\begin{align*}
\left(-\alpha_{i_1}+\alpha_{i_3}\right)\left(-\alpha_{i_4}+\alpha_\infty+1\right) \ &= \ \left(-\alpha_1+\alpha_2\right)\left(-\alpha_5+1\right) \\
&= \ \alpha_1\alpha_5-\alpha_2\alpha_5-\alpha_1+\alpha_2.
\end{align*}
\end{example}

Given this algorithm for calculating the eigenvalues of the elementary differential operators of order $m$, we may execute the sum \eqref{eq:casimir-operator} over all $(i_1,\dots,i_m)\in\{1,\dots,n\}^m$ to find the eigenvalue of the Casimir operator of order $m$ for $G$. Carrying this out for $m=2$, $3$, $4$ in Section \ref{sec:final}, we obtain \eqref{eq:casimir-eigenvalues}. One can also obtain formulas for other $m$ in this way.

Every $g \in \mathfrak{h}^n$ can be uniquely written as $g = xy$ where
\begin{equation}\label{eq:upper-half-plane2}
x \ = \ 
\left(\begin{smallmatrix}
  1 & x_{1,2} & x_{1,3} & \cdots & x_{1,n} \\
  & 1 & x_{2,3} & \cdots & x_{2,n} \\
  & &  \ddots & \ddots & \vdots \\
  & & & 1 & x_{n-1,n} \\
  & & & & 1
\end{smallmatrix}\right), \quad y \ = \ \left(\begin{smallmatrix}
y_1 y_2 \cdots y_{n-1} & & & & \\
& y_1 y_2 \cdots y_{n-2} & & & \\
& & \ddots & & \\
& & & y_1 & \\
& & & & 1
\end{smallmatrix}\right),
\end{equation}
$x_{i,j} \in \mathbb{R}$ for all $1 \leq i < j \leq n$, and $y_i > 0$ for all $1 \leq i \leq n-1$ (see Lemma \ref{lemma:Iwasawa}). 

To our knowledge, \eqref{eq:elementary-eigenvalue} is the first known explicit formula for the eigenvalues of differential operators of arbitrary order on $\mathcal{L}^2(\Gamma_n \backslash \mathfrak{h}^n)$. A characterization of this type has eluded us for a long time because the action of $\Gamma_n$ on $\mathfrak{h}^n$, as simple as it is to define, is difficult to compute in practice. To calculate the action of any $\gamma \in \Gamma_n$ on $xy \in \mathfrak{h}^n$, we need to determine the upper-triangular matrix in the Iwasawa decomposition of $\gamma xy$ to identify it with a point in $\mathfrak{h}^n$ as in \eqref{eq:upper-half-plane2}. The best-known algorithms for this are $O(n^3)$ (see, for example, our constructive proof of Lemma \ref{lemma:Iwasawa}, or the algorithm outlined in \cite[Proposition 1.2.6]{Gol06}). This is the computational complexity for a single $xy \in \mathfrak{h}^n$. Carrying this out for all $xy \in \mathfrak{h}^n$ is significantly more complex, representing the magnitude of difficulty involved in finding a general explicit formula (i.e., a $O(1)$ algorithm) for this or, for that matter, any action on $\mathfrak{h}^n$ (including that of a differential operator). To exemplify this, we list $\mathcal{D}_{3,3}$ below (as computed by Bump \cite{Bum84}): 
\begin{align}\label{eq:Casimir-order-3}
\nonumber
\mathcal{D}_{3,3} \ = \ &-y_1^2y_2\frac{\partial^3}{\partial y_1^2 \partial y_2} + y_1y_2^2\frac{\partial^3}{\partial y_1 \partial y_2^2} + y_1y_2^2\frac{\partial^3}{\partial x_{1,2}^2 \partial y_1} \\
\nonumber
&- y_1^3y_2^2\frac{\partial^3}{\partial x_{1,3}^2 \partial y_1}  -2y_1^2y_2x_{1,2}\frac{\partial^3}{\partial x_{2,3} \partial x_{1,3} \partial y_2} \\
\nonumber
&+ (-x_{1,2}^2 + y_2^2)y_1^2y_2\frac{\partial^3}{\partial x_{1,3}^2 \partial y_2} -y_1^2y_2\frac{\partial^3}{\partial x_{2,3}^2 \partial y_2} \\
\nonumber
&+ 2y_1^2y_2^2\frac{\partial^3}{\partial x_{2,3} \partial x_{1,2}\partial x_{1,3}} + 2y_1^2y_2^2x_{1,2}\frac{\partial^3}{\partial x_{1,2} \partial x_{1,3}^2} \\
\nonumber
&+y_1^2\frac{\partial^2}{\partial y_1^2} -y_2^2\frac{\partial^2}{\partial y_2^2} + (x_{1,2}^2+y_2^2)y_1^2\frac{\partial^2}{\partial x_{1,3}^2} \\
& + 2y_1^2x_{1,2}\frac{\partial^2}{\partial x_{2,3}\partial x_{1,3}} + y_1^2 \frac{\partial^2}{\partial x_{2,3}^2} - y_2^2\frac{\partial^2}{\partial x_{1,2}^2}.
\end{align}
In fact, in 2005, Broughan \cite[Appendix:\@ The GL(n)pack Manual]{Gol06} created a Mathematica package to derive the Casimir operator $\mathcal{D}_{m,n}$ of order $m$ and rank $n$ in Iwasawa variables. However, as the order $m$ and rank $n$ increase, the efficiency of the program diminishes notably. For reference, the number of terms in $\mathcal{D}_{3,4}$ and $\mathcal{D}_{4,4}$, as computed by the package, are $70$ and $294$, respectively. In this regard, Broughan \cite{Bro07} also confirms that ``it is difficult to determine any pattern in the degrees of the coefficients [of these operators]," raising the need for a different approach. 

Since the Casimir operators for $\mathfrak{h}^n$ have been so difficult to derive with current numerical methods, their eigenvalues have proved even harder to compute. While the Harish-Chandra isomorphism \cite{Har51} identifies the Casimir operators for $GL(n,\mathbb{R})$ with their eigenvalue polynomials in spectral (Langlands) variables, it does so only implicitly. We cannot apply this isomorphism to determine explicit formulas for the eigenvalues of these operators. Our methods in this paper, by contrast, provide explicit formulas for these eigenvalues. 

To summarize our approach, in Section \ref{sec:lie-theory}, we equate the problem of calculating the eigenvalue of $D_{i_1,i_2} \circ \cdots \circ D_{i_m,i_1}$ to calculating the diagonal matrix in the Iwasawa decomposition of \begin{equation}\label{eq:elementary-matrix}
A \ := \ \prod_{j=1}^m(I+t_jE_{i_j,i_{j+1}})^{-1} \ = \ (I+t_1E_{i_1,i_2})^{-1} \cdots (I+t_1E_{i_m,i_1})^{-1}.
\end{equation}
As explained before, there are numerous computational challenges associated with finding the Iwasawa decomposition of an $n \times n$ matrix. In response to them, in Section \ref{sec:graph-theory}, we consider the graph $G = G_{(i_1,\dots,i_m)}$ on $\leq m$ vertices $\{i_1, \dots, i_m\}$ with $m$ edges $\{e_1, \dots, e_m\}$, where $e_j=[i_j,i_{j+1})$ is a directed edge from $i_j$ to $i_{j+1}$ with weight
\begin{equation}
w_j \ = \ 
\begin{cases}
-t_j & \text{if } i_j \neq i_{j+1}, \\
-\frac{t_j}{(1+t_j)} & \text{if } i_j = i_{j+1},
\end{cases}
\end{equation}
for all $1 \leq j \leq m$.\footnote{We clarify that $G_{(i_1,\dots,i_m)}$ could have less than $m$ vertices (for instance, if $i_j = i_{j'}$ for some $j \neq j'$). However, $G_{(i_1,\dots,i_m)}$ has exactly $m$ edges; even if $e_j$ and $e_{j'}$ start and end at the same points, i.e., $i_j = i_{j'}$ and $i_{j+1} = i_{j'+1}$, we treat $e_j$ and $e_{j'}$ as different edges. In other words, the graph is allowed to have multi-edges.} The edges of $G$ are also ordered: if $j\leq k$, then we cannot traverse $e_j$ if we previously traversed $e_k$; that is, paths on $G$ necessarily traverse their edges in increasing order. In this setting, we recognize $A$ as the path matrix of $G$: the entry $A_{v,w}$ in position $(v,w)$ of $A$ is the sum of (the weights of) all the paths from $v$ to $w$ on $G$. This perspective guides our analysis of the recursive Iwasawa algorithm, as we relate the derivatives of the recursive Iwasawa variables \eqref{eq:recursive-variables} to partitions of the edge set of $G$. This graph-theoretic interpretation also underpins the combinatorial intuition necessary to demonstrate absolute cancellation, giving us an exact, rather than an asymptotic, formula in the theorem above. 

Explicit formulas of this type, by specifying relations on Langlands parameters, offer a fresh perspective on important spectral questions. For example, Miller \cite[Theorem 1.2.2]{Mil97} showed that the cuspidal eigenvalues of the Laplacian are strictly greater than $\frac{n^3 -n}{24}$, which in tandem with \eqref{eq:casimir-eigenvalues} provided strong evidence for the archimedean Ramanujan-Selberg conjecture (that Langlands parameters are always purely imaginary). It would be interesting to investigate the spectral relations suggested by our higher order explicit formulas as well. 

\addtocontents{toc}{\protect\setcounter{tocdepth}{1}}

%
%

\subsection*{Acknowledgments}

The idea for this work arose during a lecture by Dorian Goldfeld in the Spring 2025 offering of MATH 8675 at Columbia University. I am deeply grateful to Dorian Goldfeld for his continued guidance as the advisor of my work. The graph-theoretic approach in Section \ref{sec:graph-theory} is inspired in large part by the methods taught by Ivan Corwin in the Spring 2025 offering of MATH 6153 at Columbia University. I thank Ivan Corwin for many helpful discussions in this respect and for his comments on earlier drafts of this paper. This research was partially supported by the Science Research Fellows Program at Columbia University.

\addtocontents{toc}{\protect\setcounter{tocdepth}{1}}

%
%

\section{Lie theory}\label{sec:lie-theory}

\subsection{Background}

\begin{lemma}[Iwasawa decomposition of $GL(n,\mathbb{R})$]\label{lemma:Iwasawa}
The Iwasawa decomposition of $GL(n,\mathbb{R})$ is
$U(n,\mathbb{R}) \cdot D(n,\mathbb{R}) \cdot O(n,\mathbb{R})$, i.e., every $g \in GL(n,\mathbb{R})$ can be written as $g = n d k$, where $n \in U(n,\mathbb{R})$ is uniquely determined, and $d \in D(n,\mathbb{R})$ and $k \in O(n,\mathbb{R})$ are uniquely determined upto $\pm I$.
\end{lemma}

We provide a constructive proof of the existence of the Iwasawa decomposition as we use it in the proof of Theorem \ref{thm:elementary-differential-operator}. The proof of uniqueness is straightforward (refer to \cite[Proposition 1.2.6]{Gol06}).

\begin{proof}
To find the Iwasawa decomposition of any matrix $g \in GL(n,\mathbb{R})$, we can apply the Gram-Schmidt algorithm to the column vectors constituting 
\begin{equation}
g^{-1} \ = \
\begin{pmatrix}
    \mathbf{g}_1 & \mathbf{g}_2 & \cdots & \mathbf{g}_n
\end{pmatrix}.
\end{equation}
This process yields a list of orthogonal vectors $(\mathbf{h}_1,\ldots,\mathbf{h}_n)$, where 
\begin{equation}\label{eq:gram_schmidt_general}
    \mathbf{h}_v \ = \ \mathbf{g}_v - \sum_{k=1}^{v-1}\frac{\langle \mathbf{g}_v, \mathbf{h}_k \rangle}{\langle \mathbf{h}_k, \mathbf{h}_k\rangle}\mathbf{h}_k, \quad v=1,\dots,n,
\end{equation}
which in turn provides the components of the $QR$ decomposition of $g^{-1}$:
\begin{align} \label{eq:iwasawa-orthogonal}
q \ &= \
\begin{pmatrix}
    \frac{\mathbf{h}_1}{\langle \mathbf{h}_1, \mathbf{h}_1 \rangle^{1/2}} & \frac{\mathbf{h}_2}{\langle \mathbf{h}_2, \mathbf{h}_2 \rangle^{1/2}} & \cdots & \frac{\mathbf{h}_n}{\langle \mathbf{h}_n, \mathbf{h}_n \rangle^{1/2}}
\end{pmatrix}, \\
\label{eq:iwasawa-upper-triangular}
r \ &= \ 
\begin{pmatrix}
    \langle \mathbf{h}_1, \mathbf{h}_1 \rangle^{1/2} & \frac{\langle \mathbf{g}_2, \mathbf{h}_1 \rangle}{\langle \mathbf{h}_1, \mathbf{h}_1 \rangle^{1/2}} & \cdots & \frac{\langle \mathbf{g}_n, \mathbf{h}_1 \rangle}{\langle \mathbf{h}_1, \mathbf{h}_1 \rangle^{1/2}} \\
    & \langle \mathbf{h}_2, \mathbf{h}_2 \rangle^{1/2} & \cdots & \frac{\langle \mathbf{g}_n, \mathbf{h}_2 \rangle}{\langle \mathbf{h}_2, \mathbf{h}_2 \rangle^{1/2}} \\
    & & \ddots & \vdots \\
    & & & \langle \mathbf{h}_n, \mathbf{h}_n \rangle^{1/2}
\end{pmatrix}.
\end{align}
Taking the inverse of the $QR$ decomposition of $g^{-1}$ provides the Iwasawa decomposition of 
\begin{equation}
g \ = \ (g^{-1})^{-1} \ = \ (qr)^{-1} \ = \ r^{-1}q^{-1},
\end{equation}
in that $r^{-1}$ is upper-triangular, and $q^{-1}$ is orthogonal. 
\end{proof}

As recognized by Gelfand, Graev, and Piatetski-Shapiro \cite{GGP69}, the Iwasawa decomposition gives a concrete representation of the generalized upper half-plane associated to $G$.  
\begin{definition}\label{def:upper-half-plane}
The generalized upper half-plane $\mathfrak{h}^n$ associated to $G$ is identified with the following coset space:
\begin{align*}
& G / (K \times Z(G)) \\
&= \ \left\{
\left. \left(
\begin{smallmatrix}
1 & x_{1,2} & \cdots & x_{1,k} \\
& 1 & \cdots & x_{2,k} \\
& & \ddots & \vdots & \\
& & & 1
\end{smallmatrix}
\right) 
\left(
\begin{smallmatrix}
Y_1 & & & & \\
& Y_2 & & & \\
& & \ddots & & \\
& & & yY_{n-1} & \\
& & & & 1
\end{smallmatrix}
\right) \ \right| \ 
\substack{x_{i,j} \in \mathbb{R}, \ Y_i > 0 \\ \\ \text{ for all } 1 \leq i < j \leq n}
\right\}.
\end{align*}
\end{definition}

\begin{definition}\label{def:differential-operator}
Given any $\alpha \in \mathfrak{gl}(n,\mathbb{R})$, the differential operator $D_\alpha$ is defined for all smooth functions $F:G \to \mathbb{C}$ and $g \in G$ by 
$$D_\alpha F(g) \ := \ \left.\frac{\partial}{\partial t}F(g \cdot (I+t\alpha))\right|_{t=0},$$ where $I \in G$ is the identity matrix.   
\end{definition}
The set $$\left\{\left.D_\alpha\in U(\mathfrak{g})\right|\alpha\in\mathfrak{g}\right\}$$ generates the universal enveloping algebra $U(\mathfrak{g})$ as an associative algebra over $\mathbb{R}$. The center $Z(U(\mathfrak{g}))$ of this algebra is highly relevant to the theory of automorphic forms, in that it corresponds to the algebra of differential operators that are invariant under (commute with) the group action of $\Gamma_n$. To this end, Capelli \cite{Cap87} proved that $Z(U(\mathfrak{g}))$ is generated by the Casimir operators of order $m$, where $1 \leq m \leq n$.

\begin{definition}\label{def:casimir-operator}
Let $m$ be an integer satisfying $1 \leq m \leq n$. The Casimir operator of order $m$ for $G$ is
\begin{equation}
\mathcal{D}_{m,n} \ := \ \sum_{(i_1,\dots,i_\ell) \in [k]^{\ell}}D_{i_1,i_2} \circ \cdots \circ D_{i_{\ell-1},i_\ell} \circ D_{i_\ell,i_1},
\end{equation}
where $E_{i,j}$ is the $k\times k$ matrix with entry $(i,j)$ as $1$ and zeros everywhere else, and $D_{i,j}$ is our shorthand for the differential operator $D_{E_{i,j}}$ (in the sense of Definition \ref{def:differential-operator}). 
\end{definition}

As stated earlier, the motivation for this Lie theoretic discussion is the spectral theoretic principle that the eigenfunctions of the invariant differential operators in $Z(U(\mathfrak{g}))$ generate the space of automorphic functions on $\Gamma_n \backslash \mathfrak{h}^n$. In this spirit, we define the quintessential simultaneous eigenfunction of these invariant differential operators: the power function for the Borel subgroup of $G$. 

\begin{definition}\label{def:power-function}
Let $\alpha=(\alpha_i)_{1 \leq i \leq n} \in \mathbb{C}^n$ satisfying $\sum_{i=1}^k \alpha_i=0$. The power function $|\cdot|^{\alpha}:\mathfrak{h}^n\to\mathbb{C}$ for the Borel subgroup of $GL(k,\mathbb{R})$ with Langlands parameters $\alpha$ is defined for all $g = xy \in \mathfrak{h}^k$ by 
\begin{align}
\label{eq:power-function}
|xy|^\alpha \ = \ \left|\left(\begin{smallmatrix}
    1 & \cdots & x_{1,n-1} & x_{1,n} \\
    & \ddots & \vdots & \vdots \\
    \\
    & & 1 & x_{n-1,n} \\
    & & & 1 \\
\end{smallmatrix}\right)
\left(\begin{smallmatrix}
    Y_1 & & & \\
    & \ddots & & \\
    & & Y_{n-1} & \\
    & & & Y_n
\end{smallmatrix}\right)\right|^\alpha \ := \ \prod_{i=1}^n Y_i^{\alpha_i}.
\end{align}
\end{definition}

The phase shift $\rho^{(k)} = (\frac{k+1}{2} - i)_{1 \leq i \leq k}$ is commonly associated with the power function in keeping with the convention of the theory of root systems.

\subsection{Setup} 
Fix $(i_1, \dots, i_m) \in \{1, \dots, n\}^m$. Let $i_{m+1} := i_1$,
\begin{align}
\mathbf{0} \ := \ (0, \dots, 0), \ \ \mathbf{t} \ := \ (t_1, \dots, t_m) \in \mathbb{R}^m,
\end{align}
and 
\begin{align}
A \ := \ (I+t_mE_{i_m,i_1}) \cdots (I+t_1E_{i_1,i_2}) \in \mathbb{R}^{n \times n}.
\end{align}
For all $1 \leq v, w \leq n$, let $\mathbf{a}_v : \mathbb{R}^m \to \mathbb{R}^n$ be the $v$\textsuperscript{th} column of $A^{-1}$, and let $a_{v,w} : \mathbb{R}^m \to \mathbb{R}$ be the $(v,w)$\textsuperscript{th} entry of $A^{-1}$:
\begin{align}
\nonumber
A^{-1} \ &= \ \ (I+t_1E_{i_1,i_2})^{-1}\cdots(I+t_mE_{i_m,i_1})^{-1} \\
\nonumber
& = \ 
\begin{pmatrix}
\mathbf{a}_1(\mathbf{t}) & \cdots & \mathbf{a}_n(\mathbf{t})
\end{pmatrix} \\
\label{eq:matrixA}
& = \
\begin{pmatrix}
a_{1,1}(\mathbf{t}) & \cdots & a_{1,n}(\mathbf{t}) \\
\vdots & \ddots & \vdots \\
a_{n,1}(\mathbf{t}) & \cdots & a_{n,n}(\mathbf{t})
\end{pmatrix}.
\end{align} 
We write $(\mathbf{b}_1(\mathbf{t}),\dots,\mathbf{b}_n(\mathbf{t}))$ for the list of orthogonal vectors obtained by applying the Gram-Schmidt process to $(\mathbf{a}_1(\mathbf{t}),\dots,\mathbf{a}_n(\mathbf{t}))$.

Next, let 
\begin{align}
\ell \ := \ | \{i_1, \dots, i_m\} | \ \leq \ m.
\end{align}
Define the relative ordering on $i_1, \dots, i_m$ by
\begin{align}
\nonumber
\rho: \{1, \dots, m\} \ & \to \ \{1, \dots, \ell\}, \\
j \ & \mapsto \ \left|\{i_{j'} \mid j' \in \{1, \dots, m\}, \ i_{j'} \leq i_j\}\right|;
\end{align}
and an inverse mapping by
\begin{align}
\nonumber
\sigma: \{1, \dots, \ell\} \ & \to \ \{1, \dots, m\}, \\
j \ & \mapsto \ \min\left(\rho^{-1}(\{j\})\right).
\end{align}
Let $A'$ be the $\ell \times \ell$ submatrix of $A$ consisting of rows and columns $i_1, \dots, i_m$:
\begin{align}
A' \ := \ (I+t_mE_{\rho(m),\rho(1)}) \cdots (I+t_1E_{\rho(1),\rho(2)}).
\end{align}
For all $1 \leq v, w \leq \ell$, let $\mathbf{a}_v' : \mathbb{R}^m \to \mathbb{R}^\ell$ be the $v$\textsuperscript{th} column of $(A')^{-1}$, and let $a_{v,w}' : \mathbb{R}^m \to \mathbb{R}$ be the $(v,w)$\textsuperscript{th} entry of $(A')^{-1}$:
\begin{align}
\nonumber
(A')^{-1} \ &= \ \ (I+t_1E_{\rho(1),\rho(2)})^{-1} \cdots (I+t_mE_{\rho(m),\rho(1)})^{-1} \\
\nonumber
& = \ 
\begin{pmatrix}
\mathbf{a}_1'(\mathbf{t}) & \cdots & \mathbf{a}_\ell'(\mathbf{t})
\end{pmatrix} \\
\label{eq:matrixA'}
& = 
\begin{pmatrix}
a_{1,1}'(\mathbf{t}) & \cdots & a_{1,\ell}'(\mathbf{t}) \\
\vdots & \ddots & \vdots \\
a_{\ell,1}'(\mathbf{t}) & \cdots & a_{\ell,\ell}'(\mathbf{t})
\end{pmatrix},
\end{align} 
We write $(\mathbf{b}_1'(\mathbf{t}),\dots,\mathbf{b}_n'(\mathbf{t}))$ for the list of orthogonal vectors resulting from applying the Gram-Schmidt process to $(\mathbf{a}_1'(\mathbf{t}),\dots,\mathbf{a}_\ell'(\mathbf{t}))$.

Finally, for all $1 \leq v,w \leq \ell$, define $F_{v,w} : \mathbb{R}^m \to \mathbb{R}$ by 
\begin{align}
\nonumber
F_{v,w}(\mathbf{t}) \ &:= \ \begin{cases}
\langle \mathbf{b}_v'(\mathbf{t}), \mathbf{a}_w'(\mathbf{t}) \rangle & \text{if } v < w \\
\langle \mathbf{b}_v'(\mathbf{t}), \mathbf{b}_v'(\mathbf{t}) \rangle & \text{if } v = w \\
\langle \mathbf{a}_w'(\mathbf{t}), \mathbf{b}_v'(\mathbf{t}) \rangle & \text{if } v > w
\end{cases} \\
\nonumber
&= \ \langle \mathbf{a}_w'(\mathbf{t}), \mathbf{a}_w'(\mathbf{t}) \rangle - \sum_{k=1}^{v-1} \frac{\langle  \mathbf{a}_w'(\mathbf{t}), \mathbf{b}_k'(\mathbf{t}) \rangle \langle \mathbf{b}_k'(\mathbf{t}), \mathbf{a}_v'(\mathbf{t}) \rangle}{\langle \mathbf{b}_k'(\mathbf{t}), \mathbf{b}_k'(\mathbf{t}) \rangle}  \\
\label{eq:recursive-variables}
&= \ \langle \mathbf{a}_w'(\mathbf{t}), \mathbf{a}_w'(\mathbf{t}) \rangle - \sum_{k=1}^{v-1} \frac{F_{w,k}(\mathbf{t}) F_{k,v}(\mathbf{t})}{F_{k,k}(\mathbf{t})}.
\end{align} 

\begin{proposition}
The eigenvalue of $D_{i_1,i_2}\circ\cdots\circ D_{i_m,i_1}$ associated with $|\cdot|^\alpha$ is
\begin{align}
\nonumber
&\left.\frac{\partial^m}{\partial t_1\cdots\partial t_m}\left(\prod_{v=1}^\ell \ \langle \mathbf{b}_v'(\mathbf{t}) ,\mathbf{b}_v'(\mathbf{t}) \rangle^{-\frac{\alpha_{i_{\sigma(v)}}}{2}}\right)\right\rvert_{\mathbf{t}=\mathbf{0}}\\
\label{elementary-eigenvalue-3}
= \ & \left.\frac{\partial^m}{\partial t_1\cdots\partial t_m}\left(\prod_{v=1}^\ell \ F_{v,v}(\mathbf{t})^{-\frac{\alpha_{i_{\sigma(v)}}}{2}}\right)\right\rvert_{\mathbf{t}=\mathbf{0}}.
\end{align} 
\end{proposition}

\begin{proof}
Acting on $|\cdot|^\alpha$ by $D_{i_1,i_2}\circ\cdots\circ D_{i_m,i_1}$ at any $g =xy \in \mathfrak{h}^n$, we get
\begin{align}
\nonumber
    & D_{i_1,i_2} \circ \cdots \circ D_{i_m,i_1} |g|^{\alpha} \\
    = \ & \left. \frac{\partial^m}{\partial t_1 \cdots \partial t_m} |xy(I+t_m E_{i_m,i_1}) \cdots (I+t_1E_{i_1,i_2})|^{\alpha}\right|_{\mathbf{t}=\mathbf{0}}.
\end{align}
Let the Iwasawa decomposition of $A = (I+t_mE_{i_m,i_1}) \cdots (I+t_1E_{i_m,i_1})$ be $x'y'k'$, where $x'$, $y'$, and $k'$ are unipotent, diagonal, and orthogonal matrices, respectively.
\begin{align}
    & D_{i_1,i_2}\circ\cdots\circ D_{i_m,i_1}|g|^{\alpha} \ = \ \left. \frac{\partial^m}{\partial t_1 \cdots \partial t_m}|xyx'y'k'|^{\alpha}\right|_{\mathbf{t}=\mathbf{0}}.
\end{align}
There exists a unipotent matrix $x''$ such that $yx'=x''y$. Furthermore, since $|\cdot|^\alpha$ is invariant under the left action of unipotent matrices (such as $x$ and $x''$) and the right action of orthogonal matrices (such as $k'$), and multiplicative over the set of diagonal matrices (including $y$ and $y'$), we have
\begin{align}
    \nonumber
    D_{i_1,i_2}\circ\cdots\circ D_{i_m,i_1}|g|^{\alpha} \ &= \  \left.\frac{\partial^m}{\partial t_1\cdots \partial t_m}|xx''yy'k'|^{\alpha}\right|_{\mathbf{t}=\mathbf{0}} \\
    \nonumber
    &= \ |y|^\alpha\left.\frac{\partial^m}{\partial t_1 \cdots \partial t_m}|y'|^{\alpha}\right|_{\mathbf{t}=\mathbf{0}} \\
    &= \ |g|^{\alpha}\left. \frac{\partial^m}{\partial t_1 \cdots \partial t_m}|y'|^{\alpha}\right|_{\mathbf{t}=\mathbf{0}}.
\end{align}
Therefore, the eigenvalue of $D_{i_1,i_2}\circ\cdots\circ D_{i_m,i_1}$ associated with $|\cdot|^\alpha$ is
\begin{align}\label{eq:elementary-eigenvalue-1}
\left. \frac{\partial^m}{\partial t_1 \cdots \partial t_m}|y'|^{\alpha}\right|_{\mathbf{t}=\mathbf{0}}.
\end{align}
To calculate the diagonal component $y'$ in the Iwasawa decomposition of $A$, we follow the constructive proof of Lemma \ref{lemma:Iwasawa}. Applying the Gram-Schmidt process to the columns of $A^{-1}$, namely $\mathbf{a}_1(\mathbf{t}),\dots,\mathbf{a}_n(\mathbf{t})$, we get a list of orthogonal vectors $(\mathbf{b}_1(\mathbf{t}),\dots,\mathbf{b}_n(\mathbf{t}))$. Then
\begin{align}\label{eq:diagonal-matrix}
y' \ = \ 
\begin{pmatrix}
\langle \mathbf{b}_1(\mathbf{t}), \mathbf{b}_1(\mathbf{t}) \rangle^{-\frac{1}{2}} & & \\
& \ddots & \\
& & \langle \mathbf{b}_n(\mathbf{t}), \mathbf{b}_n(\mathbf{t}) \rangle^{-\frac{1}{2}} 
\end{pmatrix},
\end{align}
and the eigenvalue of interest \eqref{eq:elementary-eigenvalue} is
\begin{align}\label{elementary-eigenvalue-2}
&\left.\frac{\partial^m}{\partial t_1\cdots\partial t_m}\left(\prod_{v=1}^n \ \langle \mathbf{b}_v(\mathbf{t}) ,\mathbf{b}_v(\mathbf{t}) \rangle^{-\frac{\alpha_v}{2}}\right)\right\rvert_{\mathbf{t}=\mathbf{0}}.
\end{align}

Note that for all $1 \leq j \leq m$, 
\begin{align}
(I+t_jE_{i_j,i_{j+1}})^{-1} \ = \ 
\begin{cases}
I - t_j E_{i_j,i_{j+1}} & \text{if } i_j \neq i_{j+1}, \\
I - \frac{t_j}{1+t_j} E_{i_j,i_{j+1}} & \text{if } i_j = i_{j+1}.
\end{cases}
\end{align}
.In this view, if $v\neq i_j$ for all $1 \leq j \leq m$, or $w\neq i_j$ for all $1 \leq j \leq m$, then $a_{v,w}(\mathbf{t})=\delta_{v,w}$ is the Kronecker delta. Therefore, if $v \neq i_j$ for all $1 \leq j \leq m$, then $\mathbf{b}_v(\mathbf{t})$ is the $v$\textsuperscript{th} standard vector in $\mathbb{R}^n$, and the $v$\textsuperscript{th} component of $\mathbf{b}_w(\mathbf{t})$ is $0$ for all $1 \leq w \leq n$, $w \neq v$. Based on this observation, we can replace $A$ with $A'$, the submatrix of $A$ consisting only of rows and columns $i_1,\dots,i_m$ (ordered relatively). This removes the dependence of our calculation on $n$ and also demonstrates its dependence only on the relative ordering $\rho$ of $i_1, \dots, i_m$. Rewriting the target expression \eqref{elementary-eigenvalue-2} with $A'$ in place of $A$ gives \eqref{elementary-eigenvalue-3}.

\section{Graph theory}\label{sec:graph-theory}

\subsection{Background}

Let $G = G_{(i_1, \dots, i_m)}$ be a graph on $\leq m$ vertices $\{i_1, \dots, i_m\}$ with $m$ edges $\{e_1, \dots, e_m\}$, where $e_j=[i_j,i_{j+1})$ is a directed edge from $i_j$ to $i_{j+1}$ with weight
\begin{equation}
w_j \ = \ 
\begin{cases}
-t_j & \text{if } i_j \neq i_{j+1}, \\
-\frac{t_j}{1+t_j} & \text{if } i_j = i_{j+1},
\end{cases}
\end{equation}
for all $1 \leq j \leq m$. The graph is allowed to have multi-edges: if $e_j$ and $e_{j'}$ start and end at the same points (i.e., $i_j = i_{j'}$ and $i_{j+1} = i_{j'+1}$) for some $j \neq j'$, we still consider $e_j$ and $e_{j'}$ as different edges. The edges are also ordered: if $j \leq k$, then we cannot traverse $e_j$ if we previously traversed $e_k$; in other words, paths on $G$ (see Definition \ref{def:path}) necessarily traverse its edges in increasing order. 

While the following definitions apply for general directed, edge-ordered, weighted graphs, we content ourselves with stating them in the specificity of $G$. 

\begin{definition}[vertex set]\label{def:vertex-set-main}
Let $S = \{j_1, \dots, j_k\}$ be a non-empty subset of $\{1, \dots, m\}$. We call \begin{equation}\label{eq:vertex-set-main}
\nu(S) \ = \ \left\{i_{j_1}, \ i_{j_1+1}, \ i_{j_2}, \ i_{j_2+1}, \ \dots, \ i_{j_k}, \ i_{j_k+1}\right\}
\end{equation} 
the \textit{vertex set of $S$}, and denote the first and second minimum elements of $\nu(S)$ by $v_1(S)$ and $v_2(S)$, respectively. In the case that $|\nu(S)| = 1$, we adopt the convention $v_2(S) = 0$. 
\end{definition}

\begin{definition}[indegree, outdegree]\label{def:diff-degree}
Consider any $v \in \{i_1, \dots, i_m\}$, and $S =  \subset \{1, \dots, m\}$. The indegree (outdegree) of $v$ in $S$ is the number of edges that go enter (leave) $v$ in $S$:
\begin{align}
\text{in}_v(S) \ &:= \ \# \{j \in S \ | \ i_{j+1} = v\}, \\
\text{out}_v(S) \ &:= \ \# \{j \in S \ | \ i_j = v\}, \\
\Delta_v(S) \ &:= \ \text{in}_v(S) - \text{out}_v(S). 
\end{align}
\end{definition}

\begin{definition}[path, cycle]\label{def:cycle-main}
Let $S = \{j_1 < j_2 < \cdots < j_k\}$ be a non-empty subset of $\{1, \dots, m\}$. We call $S$ a \textit{path} from $v$ to $w$ if $e_{j_1} \to e_{j_2} \to \cdots \to e_{j_k}$ defines a path from $v$ to $w$ on $G$, i.e.,
\begin{align}
\label{eq:matrix-multiplication}
& E_{i_{j_1},i_{j_1+1}} E_{i_{j_2},i_{j_2+1}} \cdots E_{i_{j_k},i_{j_k+1}} \ = \ E_{v,w} \\
\label{eq:path}
\iff \ & v \ = \ i_{j_\ell}, \ \ i_{j_\ell+1} \ = \ i_{j_{\ell+1}} \ \text{ for all } \ 1 \leq \ell \leq k-1, \ \text{ and } \ i_{j_k+1} \ = \ w, 
\end{align}
where $E_{i,j}$ is the $n \times n$ elementary matrix with $1$ in position $(i,j)$ and $0$ everywhere else. In the case $v = w$, we also refer to $S$ as a cycle around $v$. We establish the convention that the empty path $\emptyset$ is a cycle around $v$ for all $1 \leq v \leq n$. 
\end{definition}

\begin{definition}[proper cycle]\label{def:proper-cycle-main}
We call $S = \{j_1 < \cdots < j_k\} \subset \{1, \dots, m\}$ a proper cycle around $v$ if it is a cycle around $v$ and $i_{j_1} < i_{j_\ell}$ for all $2 \leq \ell \leq k$. 
\end{definition}

\begin{definition}[weight of a path]\label{def:path} 
Let $S \subset \{1,\dots,m\}$ be a path in $G$. The weight of $S$ is the product of the weights of the edges in $S$ (the empty product being $1$):
\begin{equation}\label{eq:weight}
w(S) \ := \ \prod_{j \in S} w_j.
\end{equation}
\end{definition}

When we expand the product
\begin{equation}\label{eq:matrix-A1}
A \ = \ (I+t_1E_{i_m,i_1})^{-1}\cdots(I+t_mE_{i_1,i_2})^{-1},
\end{equation}
each term in the expansion corresponds to a simultaneous choice of $I$ or $-t_jE_{i_j, i_{j+1}}$ over all $1 \leq j \leq m$, which in turn corresponds to a unique subset $S \subset \{1, \dots, m\}$. Therefore, we have
\begin{align}
A \ = \ \sum_{S \subset \{1, \dots, m\}}\left(\frac{\prod\limits_{j\in S} \ \ \ -t_j}{\prod\limits_{\substack{j \in S \\ i_j = i_{j+1}}}1+t_j}\right)\left(\prod_{j \in S}E_{i_j,i_{j+1}}\right),
\end{align}
where the elementary matrix product is taken along the ascending order of the elements of $S$, and the empty product (corresponding to $S=\emptyset$) is the identity matrix $I$. Our graph-theoretic conception of matrix multiplication \eqref{eq:matrix-multiplication} is that \begin{equation}
\prod_{j \in S}E_{i_j,i_{j+1}} \ = \ 
\begin{cases}
E_{v,w} &\text{if } S \text{ is a path from } v \text{ to } w, \\
0 &\text{otherwise}.
\end{cases}
\end{equation}
This perspective also applies in the trivial case $S=\emptyset$ since $\emptyset$ is defined to be a cycle around $v$ for all $1 \leq v \leq n$, and the empty product is exactly $I=\sum_{v=1}^n E_{v,v}$.
This leads to the following characterization:
\begin{align}
\nonumber
& A \ = \ \sum_{1 \leq v,w \leq \ell}\left(\sum_{\substack{S\subset \{1, \dots, m\} \\ \text{is a path} \\ \text{from } v \text{ to } w}} \ \frac{\prod\limits_{j\in S} \ \ \ -t_j}{\prod\limits_{\substack{j \in S, \\ i_j = i_{j+1}}}1+t_j}\right) E_{v,w} \\
\label{eq:component-functions}
\implies \ & a_{v,w}(\mathbf{t}) \ = \ \sum_{\substack{S\subset \{1, \dots, m\} \\ \text{is a path} \\ \text{from } v \text{ to } w}} \ \frac{\prod\limits_{j\in J} \ \ \ -t_j}{\prod\limits_{\substack{j \in S, \\ i_j = i_{j+1}}}1+t_j}.
\end{align}
This interpretation of $A$ as the path matrix of $G$ (i.e., the entry in position $(v,w)$ of $A$ is the sum of the weights of all the paths from $v$ to $w$ on $G$) is at the heart of our proof. 

%
%

\subsection{Combinatorial Lemmas for Differentiation}\label{subsec:derivation}
We are interested in calculating mixed partial derivatives of multivariable functions. While this is an exercise in calculus, we approach it combinatorially in keeping with our graph-theoretic approach. The following lemma offers a combinatorial perspective on the generalized product rule for mixed partial derivatives.

\begin{lemma}\label{lemma:prod-rule}
Consider any integer $k \geq 1$. For each $1 \leq i \leq k$, let $f_i:\mathbb{R}^m\to\mathbb{C}$ be a complex-valued function in $m$ real variables $t_1,\dots,t_m$ that is at least $m$ times differentiable at the origin $\mathbf{0}$. Then the product of these $k$ functions $\prod_{i=1}^k f_i:\mathbb{R}^m\to\mathbb{C}$ is also a complex-valued function in $t_1,\dots,t_m$ that is at least $m$ times differentiable at $\mathbf{0}$. Furthermore, for any $S \subset \{1, \dots, m\}$, the mixed partial derivative of $\prod_{i=1}^k f_i$ at $\mathbf{0}$ with respect to $t_j$ ($j\in S$) is 
\begin{equation}\label{eq:prod-rule}
\frac{\partial^{|S|}\left(\prod_{i=1}^k f_i\right)}{\prod_{j \in S}\partial t_j}(\mathbf{0}) \ = \ \sum_{\sqcup_{i=1}^k S_i = S} \ \prod_{i=1}^k \  \frac{\partial^{|S_i|} f_i}{\prod_{j \in S_i }\partial t_j}(\mathbf{0}),
\end{equation}
where the sum is over ordered partitions of $S$.
\end{lemma}

That is, we can express the mixed partial derivative of a $k$-fold product of functions with respect to ``the variables in $S$" by ``distributing'' these variables across the $k$ functions and considering their mixed partial derivatives individually. 

\begin{proof}
We prove this lemma via (finite) induction on $|S|$. The base case $|S|=0$ is satisfied in that we only need to consider the trivial partition $\emptyset=\sqcup_{i=1}^k \emptyset$.  

For the inductive case, we assume that the lemma holds when $0\leq|S|=m'<m$. Take any $S\subset\{1, \dots, m\}$ such that $|S|=m'+1$, and any $j'\in S$. We know that the lemma holds for $S\backslash\{j'\}$, so
\begin{equation}
\frac{\partial^{|S|-1}\left(\prod_{i=1}^k f_i\right)}{\prod_{j \in S\backslash\{j'\} }\partial t_j}(\mathbf{0}) \ = \ \sum_{\sqcup_{i=1}^k S_i = S\backslash\{j'\}} \ \prod_{i=1}^k \  \frac{\partial^{|S_i|} f_i}{\prod_{j \in S_i }\partial t_j}(\mathbf{0}).
\end{equation}
Differentiating both sides with respect to $t_{j'}$ and applying the familiar product rule for single-variable derivatives, we get
\begin{align}\label{eq:ind-prod-rule}
\nonumber
\frac{\partial^{|S|}\left(\prod_{i=1}^k f_i\right)}{\prod_{j \in S}\partial t_j}(\mathbf{0}) \ &= \ \sum_{\sqcup_{i=1}^k S_i = S\backslash\{j'\}} \frac{\partial}{\partial t_{j'}}\left(\prod_{i=1}^k \  \frac{\partial^{|S_i|} f_i}{\prod_{j \in S_i }\partial t_j}\right)(\mathbf{0}) \\
\nonumber
&= \ \sum_{\sqcup_{i=1}^k S_i = S\backslash\{j'\}}\left(\sum_{i=1}^k \frac{\partial^{|S_{i}|+1} f_{i}}{\prod\limits_{j \in S_{i}\cup\{j'\}}\partial t_j}\prod_{i' \neq i}\frac{\partial^{|S_{i'}|} f_{i'}}{\prod_{j \in S_{i'} }\partial t_j}\right)(\mathbf{0}) \\
&= \ \sum_{\sqcup_{i=1}^k S_i = S} \ \prod_{i=1}^k \ \frac{\partial^{|S_i|} f_i}{\prod_{j \in S_i}\partial t_j}(\mathbf{0}).
\end{align}
To generate the $k$-partitions of $S$, we can generate the $k$-partitions of $S\backslash\{j'\}$; and for each partition $\sqcup_{i=1}^k S_i = S\backslash\{j'\}$, we can select any $S_{i}$ and append $j'$ to it. This recursive procedure guides the last equality above. This completes the induction. 
\end{proof}

In this vein, we also develop a combinatorial outlook on the generalized power rule for mixed partial derivatives. 

\begin{lemma}\label{lemma:power-rule}
Let $\beta \in \mathbb{C}$, and let $f:\mathbb{R}^m\to\mathbb{C}$ be a complex-valued function in $m$ real variables $t_1,\dots,t_m$ that is at least $m$ times differentiable in a neighborhood of the origin $\mathbf{0}$. Then the power function $f^\beta:\mathbb{R}^m\to\mathbb{C}$ is also a complex-valued function in $t_1,\dots,t_m$ that is at least $m$ times differentiable in some neighborhood of $\mathbf{0}$. Furthermore, for any non-empty subset $S \subset \{1, \dots, m\}$, the mixed partial derivative of $f^\beta$ at the origin with respect to $t_j$ ($j\in S$) is 
\begin{equation}\label{eq:power-rule}
\frac{\partial^{|S|}(f^\beta)}{\prod_{j \in S}\partial t_j}(\mathbf{0}) \ = \ \sum_{k=1}^{|S|} \ (\beta)_k \cdot f^{\beta-k}(\mathbf{0})\sum_{\substack{\sqcup_{i=1}^k S_i = S, \\ S_i \neq \emptyset}} \ \prod_{i=1}^k \ \frac{\partial^{|S_i|}f}{\prod_{j\in S_i}\partial t_j}(\mathbf{0}),
\end{equation}
where the sum is over unordered partitions of $S$, and $(\beta)_k:=\beta(\beta-1)\cdots(\beta-k+1)$ denotes the falling factorial. 
\end{lemma} 

\begin{proof}
Again, we prove this lemma via induction on $|S|$. The base case $|S|=1$ corresponds to the familiar power rule for single-variable derivatives. Letting $S=\{j\}\subset\{1,\dots,m\}$, we have
\begin{equation}
    \frac{\partial (f^\beta)}{\partial t_j}(\mathbf{0}) \ = \ \beta f^{\beta-1}(\mathbf{0}) \ \frac{\partial f}{\partial t_j}(\mathbf{0}).
\end{equation}

For the inductive case, we assume that the lemma holds when $1\leq|S|=m'<m$. Take any $S\subset\{1,\dots,m\}$ such that $|S|=m'+1$, and any $j'\in S$. We assumed that the lemma holds that for $S\backslash\{j'\}$, so 
\begin{equation}
\frac{\partial^{|S|} (f^\beta)}{\prod_{j \in S\backslash\{j'\}}\partial t_j}(\mathbf{0}) \ = \ \sum_{k=1}^{|S|-1} \ (\beta)_k \cdot f^{\beta-k}(\mathbf{0})\sum_{\sqcup_{i=1}^k S_i = S\backslash\{j'\}} \ \prod_{i=1}^k \ \frac{\partial^{|S_i|}f}{\prod_{j\in S_i}\partial t_j}(\mathbf{0}).
\end{equation}
Differentiating both sides with respect to $t_{j'}$ and applying the familiar product and power rules for single-variable derivatives, we get
\begin{align}
\nonumber
& \frac{\partial^{|S|} (f^\beta)}{\prod_{j \in S}\partial t_j}(\mathbf{0}) \\
\nonumber
= \ & 
\sum_{k=1}^{|S|-1} \ (\beta)_k \ \frac{\partial}{\partial t_{j'}}\left(f^{\beta-k}(\mathbf{0})\sum_{\sqcup_{i=1}^k S_i = S\backslash\{j'\}} \prod_{i=1}^k \ \frac{\partial^{|S_i|}f}{\prod_{j\in S_i}\partial t_j}\right)(\mathbf{0})\\
\label{eq:ind-power-rule1}
= \ & 
\sum_{k=1}^{|S|-1} \ (\beta)_{k+1} \cdot f^{\beta-k-1}(\mathbf{0}) \ \sum_{\substack{\sqcup_{i=1}^{k+1} S_i = S, \\ S_{k+1}=\{j'\}}} \ \prod_{i=1}^{k+1} \ \frac{\partial^{|S_i|}f}{\prod_{j\in S_i}\partial t_j}(\mathbf{0}) \\
\label{eq:ind-power-rule2}
+ \ & \sum_{k=1}^{|S|-1} \ (\beta)_k \cdot f^{\beta-k}(\mathbf{0})\sum_{\sqcup_{i=1}^k S_i = S\backslash\{j'\}} \frac{\partial}{\partial t_{j'}}\left(\prod_{i=1}^k \ \frac{\partial^{|S_i|}f}{\prod_{j\in S_i}\partial t_j}\right)(\mathbf{0})
\end{align}
Changing the index of summation in \eqref{eq:ind-power-rule1} and applying the product rule to \eqref{eq:ind-power-rule2}, we get
\begin{align}
\nonumber
& \frac{\partial^{|S|} (f^\beta)}{\prod_{j \in S}\partial t_j}(\mathbf{0}) \\
\nonumber
= \ & 
\sum_{k=2}^{|S|} \ (\beta)_k \cdot f^{\beta-k}(\mathbf{0}) \ \sum_{\substack{\sqcup_{i=1}^{k} S_i = S, \\ S_{k}=\{j'\}}} \ \prod_{i=1}^{k} \ \frac{\partial^{|S_i|}f}{\prod_{j\in S_i}\partial t_j}(\mathbf{0}) \\
\nonumber
+ \ & \sum_{k=1}^{|S|-1} \ (\beta)_k \cdot f^{\beta-k}(\mathbf{0})\sum_{\sqcup_{i=1}^k S_i = S\backslash\{j'\}} \sum_{i=1}^k \left(\frac{\partial^{|S_i|+1}f}{\prod\limits_{j\in S_i\cup\{j'\}}\partial t_j}\prod_{i'\neq i} \ \frac{\partial^{|S_{i'}|}f}{\prod_{j\in S_{i'}}\partial t_j}\right)(\mathbf{0}) \\
\label{eq:ind-power-rule3}
= \ & 
\sum_{k=2}^{|S|} \ (\beta)_k \cdot f^{\beta-k}(\mathbf{0}) \ \sum_{\substack{\sqcup_{i=1}^{k} S_i = S, \\ S_{k}=\{j'\}}} \ \prod_{i=1}^{k} \ \frac{\partial^{|S_i|}f}{\prod_{j\in S_i}\partial t_j}(\mathbf{0}) \\
\label{eq:ind-power-rule4}
+ \ & \sum_{k=1}^{|S|-1} \ (\beta)_k \cdot f^{\beta-k}(\mathbf{0})\sum_{\sqcup_{i=1}^k S_i = S} \ \prod_{i=1}^k \ \frac{\partial^{|S_{i}|}f}{\prod_{j\in S_{i}}\partial t_j}(\mathbf{0}).
\end{align}
The above equality can be justified using the same recursive argument for $k$-partitions of $S$ as in the proof of the previous lemma. Furthermore, note that the first sum \eqref{eq:ind-power-rule3} corresponds to (non-empty) partitions $\sqcup_{i=1}^k S_i$ of $S$ in which one of the partition elements is $S_k=\{j'\}$; and the second sum corresponds to partitions of $S$ in which none of the partition elements equal $\{j'\}$, i.e., the partition element containing $j'$ contains at least one other member of $S$ as well. Together, these sums account for all the partitions of $S$, so
\begin{align}
\frac{\partial^{|S|} (f^\beta)}{\prod_{j \in A}\partial t_j}(\mathbf{0}) \ = \ & \sum_{k=1}^{|S|} \ (\beta)_{k} \cdot f^{\beta-k}(\mathbf{0}) \ \sum_{\sqcup_{i=1}^{k+1} S_i = S} \ \prod_{i=1}^{k} \ \frac{\partial^{|S_i|}f}{\prod_{j\in S_i}\partial t_j}(\mathbf{0}).
\end{align}  
This completes the induction. 
\end{proof}

%
%

\subsection{Proof of Theorem 1.5}\label{subsec:main-thm}
Now, we extend the graph-theoretic interpretation \eqref{eq:component-functions} to the recursive Iwasawa variables $F_{v,w}(\mathbf{t})$.  

\begin{proposition}\label{prop:component-derivative}
For all $1 \leq v, w \leq \ell$ and $S \subset \{1, \dots, m\}$,
\begin{equation}\label{eq:A-partitions}
\frac{\partial^{|S|}x_{v,w}}{\prod_{j \in S}\partial t_j}(\mathbf{0}) \ = \ \begin{cases}
    (-1)^{|S|} & \text{if } S \text{ generates a path from } v \text{ to } w, \\
    0 & \text{otherwise}.
\end{cases}
\end{equation}
\end{proposition}
\begin{proof}
We consider
\begin{align}\label{eq:component-derivative}
\frac{\partial^{|S|}x_{v,w}}{\prod_{j \in S}\partial t_j}(\mathbf{0}) \ &= \ \sum_{\substack{T \subset \{1, \dots, m\} \\ \text{ generates a path} \\ \text{from } v \text{ to } w \text{ in } G}} \left.\frac{\partial^{|S|}}{\prod_{j \in S}\partial t_j} \left(\frac{\prod\limits_{j \in T} \ \ \ -t_j}{\prod\limits_{\substack{j \in T, \\ \rho(j) = \rho(j+1)}} 1+t_j}\right)\right|_{\mathbf{t}=\mathbf{0}}.
\end{align}
Every term of $x_{v,w}(\mathbf{t})$ corresponds to some $T \subset \{1, \dots, m\}$ that generates a path in $G$ from $v$ to $w$. If $T \neq S$, then
\begin{equation}
\left.\frac{\partial^{|S|}}{\prod_{j \in S}\partial t_j} \left(\frac{\prod\limits_{j \in T} \ \ \ -t_j}{\prod\limits_{\substack{j \in T, \\ \rho(j) = \rho(j+1)}} 1+t_j}\right)\right|_{\mathbf{t}=\mathbf{0}} \ = \ 0;
\end{equation}
however, if $T = S$, then
\begin{align}
\nonumber
&\left.\frac{\partial^{|S|}}{\prod_{j \in S}\partial t_j} \left(\frac{\prod\limits_{j \in T} \ \ \ -t_j}{\prod\limits_{\substack{j \in T, \\ \rho(j) = \rho(j+1)}} 1+t_j}\right)\right|_{\mathbf{t}=\mathbf{0}} \\
\nonumber
&= \ \prod_{\substack{j \in S, \\ \rho(j) \neq \rho(j+1)}}\left(\left.\frac{\partial}{\partial t_j}(-t_j)\right|_{t_j=0}\right)\prod_{\substack{j \in S, \\ \rho(j) = \rho(j+1)}}\left(\left.\frac{\partial}{\partial t_j}\left(\frac{-t_j}{1+t_j}\right)\right|_{t_j=0}\right) \\
&= \ (-1)^{|S|}.
\end{align}
This yields the desired equation for $\frac{\partial^{|S|}x_{v,w}}{\prod_{j \in S}\partial t_j}(\mathbf{0})$.
\end{proof}

\begin{proposition}
For all $1 \leq v, w \leq \ell$ and $S \subset \{1, \dots, m\}$,
\begin{equation}\label{eq:A-partitions}
\left.\frac{\partial^{|S|}}{\prod_{j \in S}\partial t_j} \langle \mathbf{x}_v(\mathbf{t}),\mathbf{x}_w(\mathbf{t}) \rangle\right|_{\mathbf{t} = 0} \ = \ \sum_{\substack{S_v \sqcup_\star S_w = S}} (-1)^{|S|},
\end{equation}
the sum being over partitions $S_v \sqcup S_w$ of $S$, where $S_v$ and $S_w$ generate paths in $G$ starting at the same vertex and ending at $v$ and $w$, respectively. 
\end{proposition}
\begin{proof}
The product rule for multiple partial derivatives gives
\begin{align}
\nonumber
\left.\frac{\partial^{|S|}}{\prod_{j \in S}\partial t_j} \langle \mathbf{a}_v'(\mathbf{t}),\mathbf{a}_w'(\mathbf{t}) \rangle\right|_{\mathbf{t} = 0} \ =& \ \sum_{k=1}^\ell \left.\frac{\partial^{|S|}}{\prod_{j \in S}\partial t_j} \left(x_{k,v}(\mathbf{t}) \cdot x_{k,w}(\mathbf{t})\right)\right|_{\mathbf{t} = 0} \\
=& \ \sum_{k=1}^\ell \sum_{S_1 \sqcup S_2 = S} \frac{\partial^{|S_1|}x_{k,v}}{\prod_{j \in S_1}\partial t_j}(\mathbf{0}) \cdot \frac{\partial^{|S_2|}x_{k,w}}{\prod_{j \in S_2}\partial t_j}(\mathbf{0}).
\end{align}
We compute each term in the sum applying Proposition \ref{prop:component-derivative}:  
\begin{align}
\nonumber
&\frac{\partial^{|S_1|}x_{k,v}}{\prod_{j \in S_1}\partial t_j}(\mathbf{0}) \cdot \frac{\partial^{|S_2|}x_{k,w}}{\prod_{j \in S_2}\partial t_j}(\mathbf{0})\\
&= \ \begin{cases}
    (-1)^{|S_1|+|S_2|} & \text{if } S_1 \ (\text{resp. } S_2) \text{ gen. a path from } k \text{ to } v \  (\text{resp. } w), \\
    0 & \text{otherwise}.
\end{cases}
\end{align}
This gives
\begin{align}
\nonumber
\left.\frac{\partial^{|S|}}{\prod_{j \in S}\partial t_j} \langle \mathbf{x}_v(\mathbf{t}),\mathbf{x}_w(\mathbf{t}) \rangle\right|_{\mathbf{t} = 0} \ &= \ \sum_{k=1}^\ell \sum_{\substack{S_1 \sqcup S_2 = S, \\ S_1 \ (\text{resp. } S_2) \\ \text{gen. a path from} \\ k \text{ to } v \ (\text{resp. } w)}} (-1)^{|S|} \\
&= \ \sum_{\substack{S_v \sqcup_\star S_w = S}} (-1)^{|S|},
\end{align}
the sum being over all partitions $S_v \sqcup S_w$ of $S$, where $S_v$ and $S_w$ generate paths in $G$ starting at a common vertex $1 \leq k \leq \ell$ and ending at $v$ and $w$, respectively. 
\end{proof}

Let $r \geq 2$,
\begin{align*}
C \ &= \ (c_i)_{1 \leq i \leq r-1} \in \{1, \dots, \ell\}^{r-1},\\
D \ &= \ (d_i)_{1 \leq i \leq r} \in \{1, \dots, \ell\}^r,
\end{align*}
and $S \subset \{1, \dots, m\}$. We define
\begin{align}
    \nonumber
    \mathcal{P}(C,D,S) \ &:= & \\
    &\left\{((P_i,Q_i))_{1 \leq i \leq r-1} \left| \begin{array}{l}
        S = \bigsqcup_{i=1}^{r-1} (P_i \sqcup Q_i); \\
        \forall 1 \leq i \leq r-1: \\ 
        P_i \neq \emptyset \text{ or } Q_i \neq \emptyset, \\ 
        P_i \text{ and } Q_i \text{ generate paths in } G \\ 
        \text{both starting at } c_i \text{ and ending} \\
        \text{at } d_i \text{ and } d_{i+1}, \text{ respectively.}
    \end{array} \right.\right\}.
\end{align}
Next, for all $1 \leq v,w \leq \ell$, we define
\begin{equation}
    \mathcal{P}(v,w,S) \ := \ \bigsqcup_{r=2}^{|S|} \ \bigsqcup_{\substack{C \in \{1, \dots, \ell\}^{r-1}, \\ D \in \{v\} \times \{1, \dots, \min(v,w)-1\}^{r-2} \times \{w\}}} \ \mathcal{P}(C,D,S).
\end{equation}
Every partition $P \in \mathcal{P}(v,w,S)$ is identified with its
\begin{enumerate}
    \item length $|P| = r$,
    \item start set $C(P) := C$,
    \item end set $D(P) := D$, and
    \item path set $S(P) := \{P_i, Q_i\}_{1 \leq i \leq r-1}$.
\end{enumerate}

\begin{proposition}
    For all $1 \leq v, w \leq \ell$ and $S \subset \{1, \dots, m\}$, 
    \begin{align}
        \frac{\partial^{|S|}F_{v,w}}{\prod_{j \in S} \partial t_j}(\mathbf{0}) \ = \ \sum_{P \in \mathcal{P}(C,D,S)} (-1)^{|S|+|P|}.
    \end{align}
\end{proposition}

Now, we define an equivalence relation $\sim$ on $\mathcal{P}(v,w,S)$. Consider any $P = ((P_i,Q_i)) \in \mathcal{P}(v,w,S)$ and $1 \leq i \leq |P|-1$. 
\begin{enumerate}
    \item If $Q_i \neq \emptyset$ and $P_{i+1} = \emptyset$, and if the path (from $c_i$ to $d_{i+1}$) generated by $Q_i$ precedes the path (from $c_{i+1} = d_{i+1}$ to $d_{i+2}$) generated by $Q_{i+1}$, then the first edge of the path generated by $Q_{i+1}$ is called a \textit{pivot point of type}.
    \item If $P_{i+1} \neq \emptyset$ and $Q_i = \emptyset$, and if the path (from $c_{i+1}$ to $d_{i+1}$) generated by $P_{i+1}$ precedes the path (from $c_i = d_{i+1}$ to $d_i$) generated by $P_i$, then the first edge of the path generated $P_i$ is called a \textit{pivot point of type}.
    \item If $Q_i = P_{i+1} = \emptyset$, then the first edge of the path generated by $P_i$ is called a \textit{pivot point of type}
    \item If an intermediate (i.e., not the first) edge $e$ in the path generated by $P_i$ or $Q_i$ is directed from a vertex strictly less than $\min(v,w)$, then $e$ is called a \textit{pivot point of type}.
    \item If $c_i < \min(v,w)$, i.e., the first edges of the paths generated by $P_i$ and $Q_i$ are directed from $c_i$ which is strictly less than $\min(v,w)$, then these first edges are jointly called a \textit{pivot point of type} 
\end{enumerate}

%
%

\section{Eigenvalues of the Casimir operators}\label{sec:final}

Let $\phi$ be a Maass form for $SL(n,\mathbb{Z})$ with Langlands parameters $\alpha = (\alpha_1, \dots, \alpha_n) \in \mathbb{C}^n$. The eigenvalue of
\begin{equation}
\mathcal{D}_{m,n} \ := \ \sum_{(i_1,\dots,i_m) \in \{1, \dots, n\}} D_{i_1,i_2} \circ \cdots \circ D_{i_{m-1},i_m} \circ D_{i_m, i_1}
\end{equation}
associated with $\phi$ is the equal to the eigenvalue of $\mathcal{D}_{m,n}$ associated with $|\cdot|^{\alpha + \rho}$, where $\rho_i = \frac{n+1}{2} - i$ is the canonical phase shift (from root theory). Theorem \ref{thm:elementary-differential-operator}, which we proved in Section \ref{sec:graph-theory}, explicitly gives the eigenvalue of $D_{i_1,i_2} \circ \cdots \circ D_{i_{m-1},i_m} \circ D_{i_m, i_1}$ associated with $|\cdot|^{\alpha+\rho}$ for all $(i_1, \dots, i_m) \in \{1, \dots, n\}^m$. We list these explicit formulas in the case $m=2$ and $3$ in Tables \ref{table:elementary-differential-operators2} and \ref{table:elementary-differential-operators3}, respectively. 

\begin{table}[ht]
    \centering
    \begin{tabular}{|c|c|}
        \hline
        & Eigenvalue of $D_{i_1,i_2} \circ D_{i_2,i_1}$ for $|\cdot|^{\alpha+\rho}$ \\
        \hline
        $i_1 > i_2$ & 0 \\
        $i_1 = i_2$ & $\left(\alpha_{i_1}+\frac{n+1}{2} - i_1\right)^2$ \\
        $i_1 < i_2$ & $-\alpha_{i_1}+\alpha_{i_2}+i_1-i_2$ \\
        \hline
    \end{tabular}
    \vspace{2mm}
    \caption{Eigenvalues of the elementary differential \\ operators of order $2$ for $\mathfrak{h}^n$}
    \label{table:elementary-differential-operators2}
\end{table}

\begin{table}[ht]
    \footnotesize
    \centering
    \begin{tabular}{|c|c|}
        \hline
        & Eigenvalue of $D_{i_1,i_2} \circ D_{i_2,i_3} \circ D_{i_3,i_1}$ for $|\cdot|^{\alpha+\rho}$ \\
        \hline
        $i_1 > i_2$ & 0 \\
        $i_1 > i_3$ & 0 \\
        $i_1 < i_2 < i_3$ & $\alpha_{i_1}-\alpha_{i_2}-i_1+i_2$ \\
        $i_1 < i_3 < i_2$ & $\alpha_{i_1}-\alpha_{i_3}-i_1+i_3$ \\
        $i_1 = i_2 < i_3$ & $\left(\alpha_{i_1}+\frac{n+1}{2}-i_1\right)\left(-\alpha_{i_1}+\alpha_{i_3}+i_1-i_3\right)$ \\
        $i_1 = i_3 < i_2$ & $\left(\alpha_{i_1}+\frac{n+1}{2}-i_1\right)\left(-\alpha_{i_1}+\alpha_{i_2}+i_1-i_2\right)$ \\
        $i_1 < i_2 = i_3$ & $\alpha_{i_1}-\alpha_{i_2}-i_1+i_2+\left(\alpha_{i_1}+\frac{n+1}{2}-i_1\right)\left(-\alpha_{i_1}+\alpha_{i_2}+i_1-i_2\right)$ \\
        $i_1 = i_2 = i_3$ & $\left(\alpha_{i_1}+\frac{n+1}{2}-i_1\right)^3$ \\
        \hline
    \end{tabular}
    \vspace{2mm}
    \caption{Eigenvalues of the elementary differential \\ operators of order $3$ for $\mathfrak{h}^n$}
    \label{table:elementary-differential-operators3}
\end{table}

We recall the following relations on Langlands parameters:
\begin{align}\label{eq:symmetric-polynomials}
&\sum_{i=1}^n \alpha_i = 0, \quad \left(\sum_{i=1}^n \alpha_i\right)^2 = 0 \ \implies \ \sum_{1 \leq i_1 < i_2 \leq n}\alpha_{i_1}\alpha_{i_2} = -\frac{1}{2}\sum_{i=1}^n\alpha_i^2.
\end{align}
We also recall the following formulas for the sum of first $n$ squares/cubes:
\begin{align}
&\sum_{j=1}^n j^2 \ = \ \frac{n(n+1)(2n+1)}{6}, \\
&\sum_{j=1}^n j^3 \ = \ \frac{n^2(n+1)^2}{4}.
\end{align}
Bearing these relations in mind, we sum the eigenvalues in Table \ref{table:elementary-differential-operators2} over all $(i_1,i_2) \in \{1, \dots, n\}^2$ to get the eigenvalue of $\mathcal{D}_{2,n}$:
\begin{equation}
\sum_{i=1}^n \alpha_i^2 - \frac{n^3-n}{12}.
\end{equation}

Similarly, summing the eigenvalues in Table \ref{table:elementary-differential-operators3} over all $(i_1,i_2,i_3) \in \{1, \dots, n\}^3$, we obtain our explicit formula for the eigenvalue of the Casimir operator $\mathcal{D}_{3,n}$ of order $3$ for $\mathfrak{h}^n$ associated with $\phi$, as given in Theorem \ref{thm:opening}:
\begin{equation}
\sum_{i=1}^n \alpha_i^3 - \frac{n}{2}\sum_{i=1}^n \alpha_i^2 + \frac{n^4-n^2}{24}.
\end{equation}
\end{proof}

\section*{References}

\printbibliography[heading=none] 
\end{document}